\documentclass[a4paper]{amsart}
\usepackage{amsmath, amssymb, amsfonts, amscd}
\theoremstyle{plain}
 \newtheorem{thm}{Theorem}
 
 \newtheorem{lem}[thm]{Lemma}
 
 \newtheorem{prob}[thm]{Problem}
\theoremstyle{definition}
 
 \newtheorem{exmp}[thm]{Example}
\theoremstyle{remark}

\title{On convergence of basic hypergeometric series}
\author{Toshio Oshima}
\address{Faculty of Science, Josai University,
1-1 Keyakidai, Sakado, Saitama 350-0295, Japan}
\email{t-oshima@josai.ac.jp}
\thanks{
\hspace*{12pt}Supported by Grant-in-Aid for Scientific Researches (B), 
No.\ 25287017, Japan Society of Promotion of Science}
\begin{document}
\begin{abstract}
We examine the convergence of $q$-hypergeometric series when $|q|=1$.
We give a condition so that the radius of the convergence is positive 
and get the radius.  We also show that 
the numbers $q$ with the positive radius of the convergence
are densely distributed in the unit circle of the complex plane of $q$
and so are those with the radius $0$.
\end{abstract}
\maketitle
\section{Introduction}
Basic hypergeometric series (cf.~\cite{GR}) with the base $q$ is defined by
\begin{equation*}
 {}_r\phi_s\left[\begin{matrix}a_1,a_2,\ldots,a_r\\b_1,\ldots,b_s\end{matrix};q,z\right]
 =\sum_{n=0}^\infty
 \frac{(a_1;q)_n(a_2;q)_n\cdots(a_r;q)_n}
 {(q;q)_n(b_1;q)_n\cdots(b_s;q)_n}\bigl((-1)^nq^{\frac{n(n-1)}2}\bigr)^{s+1-r}z^n,
\end{equation*}
where
\begin{equation*}
 (a;q)_n=\prod_{j=1}^n(1-aq^{j-1})
\end{equation*}
is the $q$-Pochhammer symbol. 
Here $a_1,\ldots,a_r$, $b_1,\ldots,b_s$ and $q$ are complex parameters.
In this paper we always assume
\begin{equation}\label{eq:As1}
 a_iq^n\ne1\text{ \ and \ } b_jq^n\ne1\qquad(i=1,\ldots,r,\ j=1,\ldots,s,\ n=0,1,2,\ldots)
\end{equation}
so that the factors $(a_i;q)_n$ and $(b_j;q)_n$ in the terms 
of the series are never zero.

Let $v_n$ be the terms of the series ${}_r\phi_s$ which contain $z^n$. 
Then we have
\begin{align*}
 \frac{v_{n+1}}{v_n}
 &=\frac{(1-a_1q^n)(1-a_2q^n)\cdots(1-a_rq^n)}
  {(1-q^{n+1})(1-b_1q^n)\cdots(1-b_sq^n)}(-q^n)^{1+s-r}z\\
 &=\frac{(a_1-q^{-n})(a_2-q^{-n})\cdots(a_r-q^{-n})}
  {(1-q^{-n-1})(b_1-q^{-n})\cdots(b_s-q^{-n})}\frac zq.
\end{align*}

If $0<|q|<1$, the radius of convergence of the series ${}_r\phi_s$ equals $\infty$ 
if $r\le s$ and equals $1$ if $r=s+1$.
If $|q|>1$ and 
\begin{equation}\label{eq:As2}
 a_1\cdots a_rb_1\cdots b_s\ne0,
\end{equation}
the radius of convergence of the series equals
\begin{equation}
 \frac{|b_1b_2\cdots b_sq|}{|a_1a_2\cdots a_r|}.
\end{equation}

In this paper we discuss the convergence of the series when $|q|=1$.
The convergence of ${}_2\phi_1$ is assumed in \cite{OS} but it is a subtle problem depending 
on the base
\begin{equation}
  q=e^{2\pi i\theta}.
\end{equation}
We assume that $\theta$ is not a rational number, namely, 
$\theta\in \mathbb R\setminus\mathbb Q$ so that $(q;q)_n$ never vanish
and we have the following theorem.
\begin{thm}\label{thm:1}
Retain the notation above and assume the conditions \eqref{eq:As1} and \eqref{eq:As2}.

{\rm i)\ } Assume that there exists a positive number $C$ such that
\[
  \left|\theta-\frac km\right|>\frac{C}{m^2}\qquad(\forall k\in\mathbb Z,\ m=1,2,3,\ldots).
\]
Then we have
\begin{equation}\label{eq:R2}
  \lim_{n\to\infty}\sqrt[n]{\bigl|\bigl(e^{2\pi i\theta};e^{2\pi i\theta}\bigr)_n\bigr|}=1.
\end{equation}

Suppose moreover that every parameter $a_i$ or $b_j$ has an absolute value different from \/ $1$
or equals $e^{2\pi i\alpha} q^\beta$ with suitable rational numbers 
$\alpha$ and $\beta$ which may depend on $a_i$ and $b_j$.
Then the radius of convergence of the series ${}_r\phi_s$ equals
\begin{equation}
 \frac
 {\max\{|b_1|,1\}\cdots\max\{|b_s|,1\}}
 {\max\{|a_1|,1\}\cdots\max\{|a_r|,1\}}.
\end{equation}

{\rm ii)\,} 
In general, we have
\begin{equation}\label{eq:max1}
 \varlimsup_{n\to\infty}\sqrt[n]{\bigl|(e^{2\pi i\theta};e^{2\pi i\theta})_n\bigr|}\le1
 \qquad(\forall\theta\in\mathbb R\setminus\mathbb Q).
\end{equation}
The set of irrational real numbers $\theta$ satisfying
\begin{equation}\label{eq:zero}
 \varliminf_{n\to\infty}\sqrt[n]{\bigl|(e^{2\pi i\theta};e^{2\pi i\theta})_n\bigr|}=0
\end{equation}
is dense in $\mathbb R$ and uncountable.
If $\theta$ satisfies \eqref{eq:zero} and the absolute value of any parameter 
$a_i$ or $b_j$ is not\/ $1$, the radius of convergence of the series ${}_r\phi_s$
equals\/ $0$.
\end{thm}

Note that it is known that an irrational number $\theta$ satisfies the assumption in 
Theorem~\ref{thm:1} i) if and only if the positive integers appearing in its expansion of
continued fraction are bounded and hence the set of real numbers $\theta$ 
satisfying it is uncountable and dense in $\mathbb R$.

Suppose $\theta\in\mathbb R\setminus\mathbb Q$ satisfies the assumption in 
Theorem~\ref{thm:1} i).
Then 
\begin{equation}\label{eq:Est3}
 \left|\frac{k_1}{m_1}\theta-\frac{k_2}{m_2}-\frac{k}m\right|
 =\left|\frac{k_1}{m_1}\right|\cdot\left|\theta-\frac{m_1(k_2m+km_2)}{k_1m_2m}\right|
 >\frac{C}{k_1m_1m_2^2\cdot m^2}
\end{equation}
for integers $k,k_1,k_2,m,m_1,m_2$ with $k_1mm_1m_2\ne0$
and therefore the number $r_1\theta +r_2$ with $r_1\in\mathbb Q\setminus\{0\}$ and 
$r_2\in\mathbb Q$ also satisfies the assumption.
\begin{exmp}
The estimate
\begin{equation}\label{eq:sqrt2}
 \Bigl|\sqrt2-\dfrac km\Bigr|>\dfrac{1}{3m^2}\quad\qquad(\forall k\in\mathbb Z,\ m=1,2,3,\ldots)
\end{equation}
shows that the real number $\theta=\sqrt2+r$ with $r\in\mathbb Q$ satisfies the assumption 
in Theorem~\ref{thm:1} i).
\end{exmp}
We will prove \eqref{eq:sqrt2}. 
We assume the existence of integers $k$ and $m$ satisfying $m\ge1$ and
$|\sqrt 2-\frac km|\le \frac 1{3m^2}$. Then we may moreover assume $m\ge2$ and therefore
\begin{align*}
  1 &\le \bigl|2m^2 - k^2\bigr|=\bigl|\bigl(\sqrt2m-k\bigr)
      \bigl(\sqrt2m+k\bigr)\bigr|\\
    &=\left|m^2\Bigl(\sqrt 2 - \frac km\Bigr)
       \right|\cdot\left|2\sqrt2 -\Bigl(\sqrt2-\frac km\Bigr)
      \right|\\
    &\le \frac13\left(2\sqrt2 +\frac1{3m^2}\right)
     \le\frac{2\sqrt2}3+\frac1{9\cdot 4}=0.9705\cdots<1,
\end{align*}
which leads a contradiction.

We will show Theorem~\ref{thm:1} i) in \S\ref{S:Es1} and 
Theorem~\ref{thm:1} ii) in \S\ref{S:Es2}.

\section{Preliminary results}
First we review the following theorem which claims that 
$k\theta\mod \mathbb Z$ for
$k=1,2,\ldots\,$ are uniformly distributed on $\mathbb R/\mathbb Z$.
\begin{thm}[Bohl, Sierpi\'nski and Weil]\label{thm:Erg}
Let $f(x)$ be a periodic function on $\mathbb R$ with period\/ $1$.
If $f(x)$ is integrable in the sense of Riemann, then
\begin{equation}\label{eq:Erg}
 \lim_{n\to\infty}\frac1n\sum_{k=1}^nf(k\theta)=\int_0^1f(x)\,dx
\qquad(\forall \theta\in\mathbb R\setminus\mathbb Q).
\end{equation}
\end{thm}
This theorem is proved by approximating $f(x)$ by a finite Fourier series
(cf.~\cite{AA}) since the theorem is directly proved if $f(x)$ is a finite 
Fourier series with the fact
\[
  \sum_{k=1}^n\frac{e^{2\pi im k\theta}}n
  =\frac{e^{2\pi im\theta}}n\left(\frac{1-e^{2\pi imn\theta}}{1-e^{2\pi im\theta}}\right)
\ \xrightarrow{n\to\infty}\ 0\qquad(m\ne0).
\]

We also prepare the following integral formula.
\begin{align}\label{eq:Int}
 \int_0^{1}\log\left|1-re^{2\pi ix}\right|dx=
 \begin{cases}
  0&(0\le r\le 1),\\
  \log r&(r\ge1).
 \end{cases}
\end{align}
The series $-\frac{\log(1-z)}z=1+\frac{z}2+\frac{z^2}3+\cdots$ converges
when $|z|<1$ and therefore
\[
  0=\int_{|z|=r}\frac{\log(1-z)}z\,dz=2\pi i\int_0^1\log(1-re^{2\pi ix})\,dx
  \qquad(0\le r<1,\ z=2\pi ix).
\]
by Cauchy's integral formula.
Since
\[
 \mathrm{Re}\int_0^1\log(1-re^{2\pi ix})\,dx=
 \int_0^1\mathrm{Re}\log(1-re^{2\pi ix})\,dx=
 \int_0^1\log|1-re^{2\pi ix}|\,dx,
\]
we have \eqref{eq:Int} when $0\le r<1$.
Moreover the relation
\[
 \log|1-re^{2\pi ix}|=\log r+\log|r^{-1}-e^{2\pi ix}|
 =\log r+\log|1-r^{-1}e^{2\pi ix}|
\]
assures \eqref{eq:Int} when $r>1$.

Note that the expansion $e^\sigma-1=\sigma(1+\frac{\sigma}{2!}+\frac{\sigma^2}{3!}+\cdots)$
assures $|1-e^\sigma|\ge \frac{|\sigma|}2$ when $|\sigma|<1$.
Hence if $r=1$, the improper integration in \eqref{eq:Int} converges
because 
\begin{equation}\label{eq:est1-e}
\bigl|\log|1-e^{2\pi iz}|\bigr|> \bigl|\log |\pi z|\bigr|
\text{ \ for \ }0<|2\pi z|<1
\end{equation}
and we obtain \eqref{eq:Int} by taking the limit $r\to 1-0$
(cf.~\cite[5.3.5]{Ah}).

\section{A lemma}
We prepare a lemma to prove Theorem~\ref{thm:1} i).
\begin{lem}\label{lem:1}
Let $f(x)$ be a periodic function on $\mathbb R$ with period\/ $1$.
Suppose that $f(x)$ is continuous on $[0,1]$ except for finite points
$c_1,\ldots,c_p\in[0,1)$.
Suppose there exist $r_j\in \mathbb Q$ for $j=1,\ldots,p$ such that
\begin{equation}\label{eq:cj}
   c_j-r_j\theta\in\mathbb Q\text{ \ and \ }
  (r_j+k)\theta-c_j\not\in\mathbb Z\text{ \ for \ }k=1,2,\ldots.
\end{equation}
Suppose moreover that there exist a positive number $\epsilon$ and a
continuous function $h(t)$ on $(0,1]$ such that
\begin{equation}\label{eq:Integrable}
 \begin{split}
  |f(x)|&<h(|x-c_j|) \quad\text{for \ }0<|x-c_j|<\epsilon,\\
  \int_0^1h(t)\,dt&<\infty\text{ \ and \ }h(t_1)\ge h(t_2)\ge 0
  \text{ \ if \ }0<t_1<t_2\le 1.
 \end{split}
\end{equation}
If $\theta\in\mathbb R\setminus\mathbb Q$ satisfies the assumption in 
Theorem~\ref{thm:1} i), then
\eqref{eq:Erg} is valid.
Here we note that the condition \eqref{eq:Integrable} assures that the improper integral 
in \eqref{eq:Erg} converges.
\end{lem}
\begin{proof}
Put
\begin{align*}
 J(j,n,\epsilon)=\bigl\{k\mid 1\le k\le n,\ 
 \min_{m\in\mathbb Z}\{|k\theta-c_j-m|\}<\epsilon\bigr\}
\intertext{and}
 I_\epsilon=\bigl\{x\in[0,1]\mid \min_{m\in\mathbb Z}|x-c_j-m|\ge\epsilon 
\text{ \ for \ }j=1,\ldots,p\bigr\}.
\end{align*}
Then Theorem~\ref{thm:Erg} shows that
\[
 \lim_{n\to\infty}\frac1n
 \sum_{\substack{k\not\in J(1,n,\epsilon)\cup\cdots\cup J(p,m,\epsilon)\\ 1\le k\le n}}f(k\theta)
 =\int_{I_\epsilon}f(x)\,dx
\]
and therefore we have only to show
\begin{equation}\label{eq:ToProve}
 \lim_{\epsilon\to+0}
 \lim_{n\to\infty}\frac1n\sum_{k\in J(j,n,\epsilon)}|f(k\theta)|=0
\end{equation}
to get this lemma.

Fix $j$.
Since Theorem~~\ref{thm:Erg} shows that
\[
 \lim_{n\to\infty}\frac1n\#J(j,n,\epsilon)=2\epsilon,
\]
there exists a positive integer $N_\epsilon$ such that
\[
 \#J(j,n,\epsilon)\le 3\epsilon n\quad(\forall n\ge N_\epsilon).
\]
Put
\[
 J(j,\epsilon,n)=\{k_1,k_2,\ldots,k_L\}
\]
with $L=\#J(j,\epsilon,n)$ so that
\[
 \min_{m\in\mathbb Z}|k_\nu\theta-c_j-m|\le \min_{m\in\mathbb Z}|k_{\nu'}\theta-c_j-m|
 \text{ \ if \ }1\le\nu<\nu'\le L.
\]

Note that $c_j=\frac{k_1}{m_1}\theta+\frac{k_2}{m_2}$ with integers 
$k_1,k_2,m_1,m_2$.
In view of \eqref{eq:Est3}, we have
\[
 \left|k\theta-\frac{k_1}{m_1}\theta-\frac{k_2}{m_2}-m\right|
 =\left|\frac{m_1k-k_1}{m_1}\theta-\frac{k_2+mm_2}{m_2}\right|>
  \frac{C}{m_1m_2^2|mk-k_1|}
\]
for $k=1,2,3,\ldots$ satisfying $mk\ne k_1$.
If $\frac{k_1}{m_1}$ is a positive integer, the assumption implies
$\frac{k_2}{m_2}\not\in\mathbb Z$.
Hence replacing $C$ by a small positive number if necessary, we may assume
\[
 \min_{m\in\mathbb Z}\left|k\theta -c_j-m\right|>\frac Ck\quad(k=1,2,\ldots,\ j=0,1,\ldots,p),
\]
where we put $c_0=0$.  In particular, we have
\[
 \min_{m\in\mathbb Z}|k\theta - k'\theta-m|>\frac C{k'-k}\ge\frac Cn\quad(0\le k<k'\le n).
\]
Thus we have
\[
 \min_{m\in\mathbb Z}\left|k_\nu\theta -c_j-m\right|>\frac{C\nu}{2n}\quad(1\le\nu\le L)
\]
and
\[
  |f(k_\nu\theta)|<h
  \bigl(\tfrac{C\nu}{2n}\bigr)\quad(1\le\nu\le L<3\epsilon n).
\]
Hence if $n\ge N_\epsilon$, we have
\begin{align*}
 \frac1n\sum_{k\in J(j,n,\epsilon)}|f(k\theta)|
 &=\frac1n\sum_{\nu=1}^L|f(k_\nu\theta)|
 \le \frac1n\sum_{\nu=1}^{[3\epsilon n]}h(\tfrac{C\nu}{2n})\\
 &\le \int_0^{\frac{[3\epsilon n]}n}h\bigl(\tfrac{Cx}2\bigr)\,dx
 \le \frac 2C\int_0^{\frac{6\epsilon} C}h(t)\,dt,
\end{align*}
which implies \eqref{eq:ToProve}.
Here $[3\epsilon n]$ denotes the largest integer satisfying $[3\epsilon n]\le 3\epsilon n$.
\end{proof}
\begin{prob} 
Let $f(x)$ be a function satisfying the assumption in Lemma~\ref{lem:1}.
Is the equality \eqref{eq:Erg} for $\theta\in\mathbb R$ valid almost everywhere
in the sense of Lebesgue measure?
Here we may assume $p=1$, $c_1=0$ and $h(t)=|\log t|$ for our problem.
Is it also valid without assuming the condition \eqref{eq:cj}?
\end{prob}
\section{Estimate I}\label{S:Es1}
Let $a=re^{2\pi i\tau}$ and $q=e^{2\pi i\theta}$ with $\tau,\,\theta\in\mathbb R$ and $r>0$.
Then
\begin{equation}\label{eq:Poch}
  \sqrt[n]{\bigl|(a;q)_n\bigr|}
  =\exp\left(\frac1n\sum_{k=0}^{n-1}\log\bigl|1-re^{2\pi i(k\theta+\tau)}\bigr|\right).
\end{equation}
If $r\ne1$, Theorem~\ref{thm:Erg} and \eqref{eq:Int} imply
\begin{equation}\label{eq:R1}
  \lim_{n\to\infty}\sqrt[n]{\bigl|(a;q)_n\bigr|}=
  \max\{|a|,1\}\qquad(|a|\ne 1,\ q=e^{2\pi i\theta}\text{\ \ with\ \ }
  \theta\in\mathbb R\setminus\mathbb Q).
\end{equation}

Assume $r=1$.
Since
\begin{align*}
 \sum_{\substack{\min_{m\in\mathbb Z}\{|k\theta-m|\}<\epsilon\\ 1\le k\le n}}
 \log\bigl|1-e^{2\pi ik\theta}\bigr|&\le 0
\intertext{and}
 \lim_{n\to\infty}
 \frac1n\sum_{\substack{\min_{m\in\mathbb Z}\{|k\theta-m|\}\ge\epsilon\\ 1\le k\le n}}
 \log\bigl|1-e^{2\pi ik\theta}\bigr|&=\int_\epsilon^{1-\epsilon}\log|1-e^{2\pi ix}|\,dx
\end{align*}
for any small positive number $\epsilon$, 
we have \eqref{eq:max1} in view of \eqref{eq:Int} and \eqref{eq:Poch}.

Now assume moreover that $\theta$ satisfies the assumption in Theorem~\ref{thm:1} i).
Then Lemma~\ref{lem:1} with $f(x)=\log|1-e^{2\pi ix}|$ and $h(t)=|\log\pi t|$ 
(cf.~\eqref{eq:est1-e}) proves \eqref{eq:R2}.

Suppose $\tau=\frac{k_1}{m_1}\theta+\frac{k_2}{m_2}$ with integers
$k_1,k_2,m_1,m_2$ with $m_1>0$, $m_2>0$ and
\[
 \left(\frac{k_1}{m_1}+k-1\right)\theta+\frac{k_2}{m_2}\notin\mathbb Z
 \qquad(k=1,2,3,\ldots)
\]
corresponding to \eqref{eq:As1}, \eqref{eq:As2} and \eqref{eq:cj}.
Lemma~\ref{lem:1} with $f(x)=\log\bigl|1-e^{2\pi i(x+(\frac{k_1}{m_1}-1)
\theta)}\bigr|$, $h(t)=|\log\pi t|$
and $c_j=-(\frac{k_1}{m_1}-1)\theta-\frac{k_2}{m_2}$ implies
\begin{equation}\label{eq:R3}
 \lim_{n\to\infty}\sqrt[n]{\bigl|\bigl(e^{2\pi i\tau};e^{2\pi i\theta}\bigr)_n\bigr|}=1.
\end{equation}
Thus we have Theorem~\ref{thm:1} i) by the estimates \eqref{eq:R1}, \eqref{eq:R2} 
and \eqref{eq:R3}.
\section{Estimate II}\label{S:Es2}
Define a series of rapidly increasing positive integers $\{a_n\}$ by
\begin{equation}
 a_1=2,\quad a_{n+1}=k_{n+1}\cdot a_n\cdot a_n!\qquad(k_{n+1}=2\text{\ or\ }3,\ n=1,2,3,\ldots)
\end{equation}
and put $\theta=\displaystyle\sum_{n=1}^\infty\frac1{a_n}$ and $q=e^{2\pi i\theta}$. 
Then $\theta\not\in\mathbb Q$ and we have
\begin{gather}
 \min_{m\in\mathbb Z}\bigl|a_n\cdot\theta-m\bigr|<\sum_{k=n+1}^\infty\frac{a_n}{a_k}<\frac1{a_n!},
 \label{eq:small}\\
 \bigl|1-e^{2\pi ia_n\theta}\bigr|\le\frac{2\pi}{a_n!},\quad
 \prod_{j=1}^{a_n}\bigl|1-q^j\bigr|\le\frac{2^{a_n}\pi}{a_n!},\quad
 \lim_{n\to\infty}\sqrt[a_n]{\prod_{j=1}^{a_n}\bigl|1-q^j\bigr|}=0\notag
\end{gather}
and \eqref{eq:zero} in Theorem~\ref{thm:1} ii).
We may choose $k_n\in\{2,3\}$ for $n=1,2,\ldots$, 
we get uncountably many $\theta$'s.
Moreover if we put $\theta=\sum_{n=1}^\infty\frac 1{a_n}+r$ so that
there exists a positive integer $N$ satisfying $rN\in\mathbb Z$, 
then $\theta$ also satisfies \eqref{eq:small} for $n\ge N$ and 
hence $\theta$ satisfies  \eqref{eq:zero}.

The remaining claim in Theorem~\ref{thm:1} ii) is clear
from \eqref{eq:R1}.

\end{document}